\documentclass[12pt,reqno]{amsart}

\usepackage{amsmath,amsfonts,amsthm}
\newtheorem{theorem}{Theorem}[section]
\newtheorem{prop}[theorem]{Proposition}

\newtheorem{definition}[theorem]{Definition}
\usepackage[a4paper,height=20cm,top=24mm,bottom=24mm,left=26mm,right=26mm]{geometry}

\usepackage[numbers,sort]{natbib}
\usepackage[breaklinks=true]{hyperref}

\usepackage[footnotesize,hang,sf]{caption}\setcaptionwidth{12cm}

\usepackage[]{graphicx}            
\usepackage{color}
                                        
\DeclareMathAlphabet{\mathpzc}{OT1}{pzc}{m}{it}
\usepackage{titlesec}
\titleformat{\section}
{\bfseries\scshape\centering}
{\thesection.}{.5em}{}
\titleformat{\subsection}
{\rmfamily\bfseries}
{\thesubsection}{.5em}{}
\titleformat{\subsubsection}
{\rmfamily\bfseries}
{\thesubsubsection}{.5em}{}
\numberwithin{equation}{section}

\newcommand{\I}{\mathrm{i}}
\newcommand{\E}{\mathrm{e}}

\DeclareMathDelimiter{\Norm}{\mathord}{largesymbols}{"3E}{largesymbols}{"3E}

\begin{document}
\baselineskip 16pt
\parskip 8pt
\sloppy


\title{Torus Knots and Quantum Modular Forms}


\author[K. Hikami]{Kazuhiro \textsc{Hikami}}

\address{Faculty of Mathematics,
  Kyushu University,
  Fukuoka 819-0395, Japan.}

\email{
  \texttt{khikami@gmail.com}
}

\author[J. Lovejoy]{Jeremy \textsc{Lovejoy}}

\address{CNRS, LIAFA,
Universite Denis Diderot - Paris 7,
Case 7014,
75205 Paris Cedex 13
France}

\email{
  \texttt{lovejoy@math.cnrs.fr}
}


\date{\today}

\begin{abstract}
In this paper we compute a $q$-hypergeometric expression for the
cyclotomic expansion of the colored Jones polynomial for the
left-handed torus knot
$(2,2t+1)$
and use this to define a family of quantum modular forms which are dual to the generalized Kontsevich-Zagier series.   
\end{abstract}



\subjclass[2010]{
}


\maketitle

\section{Introduction and Statement of Results}

Zagier introduced quantum modular
forms~\cite{Zagier09a} as analogues of modular forms
which behave nicely at roots of unity.  The word \emph{quantum} refers to the fact that 
these objects have ``the `feel' of the objects in perturbative quantum field theory'' \cite[p. 659]{Zagier09a}.   A
celebrated example is the Kontsevich--Zagier series~\cite{DZagie01a}
\begin{equation}
  \label{KZ_F}
  F(q) 
  := \sum_{n=0}^\infty (q)_n,
\end{equation}
where we use the standard notation, 
\begin{equation}
(a)_n := (1-a)(1-aq) \cdots (1-aq^{n-1}).
\end{equation}
Note that $F(q)$ does not converge on any open subset of~$\mathbb{C}$,
but it is well-defined at roots of unity.

Bryson et al ~\cite{BryOnoPitRho12a} recently established a relationship between
$F(q)$ and the generating function for strongly unimodal sequences,
\begin{equation}
  U(x;q) 
  :=
  \sum_{n=0}^\infty
  (-x \, q)_n \, (-x^{-1} \, q)_n  \, q^{n+1} .
\end{equation}
Precisely,
$U(-1;q)$ is dual to $F(q)$
at the root of unity~$\zeta_N=\E^{2\pi \I/N}$,
\begin{equation}
  \label{F_and_U}
  F(\zeta_N^{~-1}) = U(-1;\zeta_N) .
\end{equation}
Note that contrary to the Kontsevich--Zagier series $F(q)$, the function~$U(x;q)$ converges for generic~$|q| < 1$.  

This article is based on the observation that
the identity~\eqref{F_and_U}
can be interpreted in terms of quantum topology as follows.
We use the $N$-colored Jones polynomial $J_N(K;q)$
for a knot~$K$, which
is based on the $N$-dimensional representation of $U_q(s\ell_2)$
(see \emph{e.g.}~\cite{Licko97Book}).
Throughout this article we use a normalization
$
  J_N(\text{unknot};q)=1
$.
It is known that
\begin{equation}
  \label{mirror_J}
  J_N(K;q^{-1})= J_N(K^*;q)  ,
\end{equation}
where $K^*$ is a mirror image of $K$.
The colored Jones polynomial for the right-handed trefoil $T_{(2,3)}$
and the left-handed trefoil $T_{(2,3)}^*$ may be respectively written as
(see \emph{e.g.}~\cite{Habiro00a,GMasb03a,TQLe03a})
\begin{align}
  J_N(T_{(2,3)}; q)& =
  q^{1-N}\sum_{n=0}^\infty q^{- n N} (q^{1-N})_n ,
  \label{J_right_trefoil}
  \\
  J_N(T_{(2,3)}^*; q)
  &  =
  \sum_{n=0}^\infty
  q^n \,  (q^{1-N})_n \, (q^{1+N})_n .
  \label{J_left_trefoil}
\end{align}
It should be remarked that
in these $q$-hypergeometric expressions
an  infinite series terminate at finite~$n$.
One finds that the duality~\eqref{F_and_U} is a consequence
of~\eqref{mirror_J}.

Based on the colored Jones polynomial for the torus knot $T_{(2,2t+1)}$ at roots of unity,
the first author ~\cite{KHikami04a} introduced a family of quantum modular forms generalizing $F(q)$~\eqref{KZ_F},
\begin{equation}
  \label{define_Ft}
  F_t(q)
  :=
  q^{t}
  \sum_{k_t \geq \dots \geq k_1 \geq 0}^\infty
  (q)_{k_t} \, 
  \prod_{i=1}^{t-1} q^{k_i(k_i+1)} 
  \begin{bmatrix}
    k_{i+1} \\
    k_i
  \end{bmatrix}_q .
\end{equation}
Here $\begin{bmatrix}
    n \\
    k
  \end{bmatrix}_q$ is the usual $q$-binomial coefficient,

\begin{equation}
\begin{bmatrix}
    n \\
    k
  \end{bmatrix}_q := \frac{(q)_n}{(q)_{n-k}(q)_k}.
\end{equation}	
Note that when $t=1$ we recover the Kontsevich--Zagier series,
$F_{t=1}(q)=q \, F(q)$.
Our purpose in this article is to use the perspective of quantum invariants to generalize
$U(x;q)$ and~\eqref{F_and_U}.
As a dual to  $F_t(q)$, we make the following definition.
\begin{definition} \label{Utdef}
The generalized $U$-function $U_t(x;q)$ is defined by
\begin{equation}
  \label{define_Ut}
  U_t(x;q)
  :=
  q^{-t}
  \sum_{k_t  \geq \cdots \geq k_1 \geq 1}
  (-xq)_{k_t-1}(-x^{-1}q)_{k_t-1} \,
  q^{k_t}\prod_{i=1}^{t-1} q^{k_i^2}
  \begin{bmatrix}
    k_{i+1} + k_i - i   + 2\sum_{j=1}^{i-1}k_j \\
    k_{i+1} - k_i
  \end{bmatrix}_q .
\end{equation}
\end{definition}


Our first result is the following generalization of~\eqref{F_and_U}.

\begin{theorem}
  \label{thm:Ft_and_Ut}
  \begin{equation}
    F_t(\zeta_N^{~-1})
    =
    U_t(-1 ; \zeta_N) .
  \end{equation}
\end{theorem}

Our second result is a Hecke-type expansion for $U_t(x;q)$.

\begin{theorem} \label{thm:HeckeUt}
We have 
\begin{align}
  &U_t(-x;q)
  \nonumber \\
  &=
  \begin{aligned}[t]
    & -q^{-t/2-1/8} \frac{(xq)_{\infty}(q/x)_{\infty}}{(q)_{\infty}^2} 
    \\
    & \quad \times
    \Bigg(\sum_{\substack{r,s \geq 0 \\ r \not \equiv s
        \pmod{2}}} - \sum_{\substack{r,s < 0 \\ r \not \equiv s
        \pmod{2}}}
    \Bigg)
    \frac{(-1)^{\frac{r-s-1}{2}}
      q^{\frac{1}{8}r^2
        + \frac{4t+3}{4} r s
        + \frac{1}{8}s^2
        +  \frac{2+t}{2} r+\frac{t}{2} s }
    }{
      1-x q^{\frac{r+s+1}{2} }
    }
  \end{aligned}
  \label{HeckeUteq1}
  \\ 
  &=
  \begin{aligned}[t]
    & -q^{-t/2-1/8} \frac{(xq)_{\infty}(q/x)_{\infty}}{(q)_{\infty}^2} 
    \\
    &  \times
    \Bigg(\sum_{\substack{r,s,u \geq 0 \\ r \not \equiv s
        \pmod{2}}} + \sum_{\substack{r,s,u < 0 \\ r \not \equiv s
        \pmod{2}}}
    \Bigg)
    (-1)^{\frac{r-s-1}{2}} x^u
    q^{\frac{1}{8}r^2+\frac{4t+3}{4} r s +
      \frac{1}{8}s^2
      +
      \frac{2+t}{2} r
      +\frac{t}{2} s  +
      u \frac{r+s+1}{2}}. 
  \end{aligned}
  \label{HeckeUteq2}
\end{align}
\end{theorem} 



The paper is constructed as follows.  In Section 2, we review Bailey pairs and their relation to the colored Jones polynomial.  In Section 3 we study the colored Jones polynomial for the torus knot $T_{(2,2t+1)}$.  In particular, we use the Bailey pair machinery to compute the coefficients of the cyclotomic expansion of $J_N(T^*_{(2,2t+1)};q)$, which leads to Theorem~\ref{thm:Ft_and_Ut}.  In Section 4 we prove Theorem~\ref{thm:HeckeUt}, again using the Bailey machinery.  In Section 
5 we extend Theorems~\ref{thm:Ft_and_Ut} and~\ref{thm:HeckeUt} to the vector-valued setting.   We close with some suggestions for future research and an appendix containing some examples. 


\section{Bailey Pairs and The Colored Jones Polynomial}
In this section we review facts about Bailey pairs and their relation to the colored Jones polynomial.

First recall~\cite{Andre84a} that two sequences $(\alpha_n,\beta_n)$ form a Bailey pair relative to $a$ if
\begin{equation} \label{Baileypairdef1}
\beta_n = \sum_{j=0}^n \frac{\alpha_j}{(q)_{n-j}(aq)_{n+j}},
\end{equation} 
or equivalently,
\begin{equation} \label{Baileypairdef2}
  \alpha_n
  = \frac{1 - aq^{2n} }{1-a} \,
  \frac{(a)_n}{(q)_n} \,
  (-1)^nq^{n(n-1)/2}
  \sum_{j=0}^n (q^{-n})_j(aq^n)_jq^j\beta_j.
\end{equation} 
The Bailey lemma \cite{Andre84a} states that if $(\alpha_n,\beta_n)$  is a Bailey pair relative to $a$, then so is $(\alpha_n',\beta_n')$, where 
\begin{equation} \label{alphaprimedef}
\alpha'_n = \frac{(b)_n(c)_n(aq/bc)^n}{(aq/b)_n(aq/c)_n}\alpha_n
\end{equation} 
and
\begin{equation} \label{betaprimedef}
\beta'_n = \sum_{k=0}^n\frac{(b)_k(c)_k(aq/bc)_{n-k} (aq/bc)^k}{(aq/b)_n(aq/c)_n(q)_{n-k}} \beta_k.
\end{equation} 
In particular, if $b,c \to \infty$ then we have
\begin{equation} \label{rho1rho2inftyalpha}
\alpha'_n = a^nq^{n^2}\alpha_n
\end{equation}
and
\begin{equation} \label{rho1rho2inftybeta}
\beta'_n = \sum_{k=0}^n \frac{a^kq^{k^2}}{(q)_{n-k}}\beta_k.
\end{equation}
Inserting \eqref{alphaprimedef} and~\eqref{betaprimedef} back in the
definition~\eqref{Baileypairdef1} and letting $n \to \infty$, we have
\begin{equation}
\sum_{n \geq 0} (b)_n(c)_n(aq/bc)^n \beta_n = \frac{(aq/b)_{\infty}(aq/c)_{\infty}}{(aq)_{\infty}(aq/bc)_{\infty}} \sum_{n \geq 0} \frac{(b)_n(c)_n(aq/bc)^n}{(aq/b)_n(aq/c)_n} \alpha_n.
\end{equation}

Next recall the cyclotomic expansion of the colored Jones polynomial due to
Habiro~\cite{KHabiro06b}
\begin{equation}
  \label{J_and_C}
  J_N(K;q)
  =
  \sum_{n=0}^\infty
  C_n(K; q) \, (q^{1+N})_n \, (q^{1-N})_n ,
\end{equation}
where we have
\begin{equation} \label{C_nLaurent}
C_n(K; q) \in \mathbb{Z}[q,q^{-1}].
\end{equation}
The colored Jones polynomial $J_N(K;q)$ and the coefficients $C_n(K;q)$
defined in~\eqref{J_and_C} can be regarded
as a Bailey pair
$(\alpha_n,\beta_n)$ relative to $q^2$.
Namely, comparing equations~\eqref{J_and_C} and~\eqref{Baileypairdef2} we have (see also~\cite{KHabiro06b,KHikami06b})
\begin{equation}
  \label{Bailey_JC}
  \begin{aligned}
    \alpha_n
    & =
    \frac{
      (1-q^{n+1}) \, (1-q^{2n+2})
    }{(1-q) \, (1-q^2)} \,
    (-1)^n \, q^{\frac{1}{2} n (n-1)} \,
    J_{n+1}(K;q) ,
    \\[2mm]
    \beta_n & =
    q^{-n} \, C_n(K;q) .
  \end{aligned}
\end{equation}
Equation~\eqref{Baileypairdef1} gives the inverse transform 
\begin{equation}
  \label{C_from_J}
  C_n(K; q)
  =
  - q^{n+1} \sum_{\ell = 1}^{n+1}
  \frac{
    (1-q^\ell) \, (1 - q^{2 \ell})
  }{
    (q)_{n+1-\ell} \, (q)_{n+1+\ell}
  } \,
  (-1)^\ell \, q^{\frac{1}{2} \ell(\ell-3)} \,
  J_\ell(K;q) .
\end{equation}


\section{The Colored Jones Polynomial for Torus Knots}
For some  knots $K$, explicit forms of $J_N(K;q)$ and/or $C_N(K;q)$ are known in
the literature.  For instance, when $K$ is the right-handed torus knot
$T_{(s,t)}$,
where $s$ and $t$ are coprime positive integers,
the colored Jones polynomial is given by~\cite{Mort95a,RossJone93a}
\begin{equation}
  \label{J_Morton}
  J_N(T_{(s,t)}; q)
  =
  \frac{q^{\frac{1}{4} s t (1-N^2)}}{
    q^{\frac{N}{2}} - q^{-\frac{N}{2}}
  }
  \sum_{j= - \frac{N-1}{2}}^{\frac{N-1}{2}}
  q^{s t j^2} \,
  \left(
    q^{-(s+t) j + \frac{1}{2}} -
    q^{-(s-t) j - \frac{1}{2}}
  \right) .
\end{equation}
Using difference equations,
the first author ~\cite{KHikami02c} constructed a $q$-hypergeometric expression 
for $J_N(T_{(s,t)}; q)$ when $s=2$,
\begin{equation}
  J_N(T_{(2,2t+1)}; q)
  =
  q^{t(1-N)}
  \sum_{k_t \geq \dots \geq k_1 \geq 0}^\infty
  (q^{1-N})_{k_t} \, q^{-N k_t} \,
  \prod_{i=1}^{t-1} q^{k_i(k_i+1-2N)} 
  \begin{bmatrix}
    k_{i+1} \\
    k_i
  \end{bmatrix}_q .
\end{equation}
(See ~\cite{KHikami03b} for similar expressions for some other torus knots.)
Comparing this with the generalized Kontsevich--Zagier
series~\eqref{define_Ft}, we find that $J_N(T_{(2,2t+1)}; q)$ and $F_t (q)$ agree at roots of
unity,
\begin{equation}
  \label{J_and_Ft_unity}
  J_N(T_{(2,2t+1)}; \zeta_N)
  =
  F_t (\zeta_N) .
\end{equation}

With~\eqref{mirror_J} and~\eqref{J_and_C} in mind, we see that to
discover Definition~\ref{Utdef} and prove Theorem~\ref{thm:Ft_and_Ut}
we need to compute the cyclotomic expansion of the colored Jones
polynomial of the left-handed torus knot $T_{(2,2t+1)}^*$.
Recalling that the colored Jones polynomial for  the mirror image
$K^*$ is given from that for $K$~\eqref{mirror_J},
we find from~\eqref{J_Morton} that
\begin{equation}
  \label{left_hand_J}
  (1-q^N) \, J_N(T_{(2,2t+1)}^*;q)
  =
  (-1)^N q^{- t +\frac{1}{2} N + \frac{2t+1}{2} N^2 }
  \sum_{k=-N}^{N-1} (-1)^k q^{-\frac{2t+1}{2}k(k+1)+k}.
\end{equation}
Then the coefficients $C_n$ in the
cyclotomic expansion~\eqref{J_and_C} are given 
from the inverse transform~\eqref{C_from_J} as
\begin{multline}
  \label{C_for_Tt}
  C_{n-1}(T_{(2,2t+1)}^{*} ; q)
  \\
  =
  -q^{n-t} \sum_{\ell=0}^n \frac{1}{(q)_{n-\ell} \, (q)_{n+\ell}} \,
  q^{(t+1)\ell^2-\ell} \, (1 - q^{2\ell}) 
  \sum_{k=-\ell}^{\ell-1}
  (-1)^k \, 
  q^{-(t+\frac{1}{2}) k^2 - (t - \frac{1}{2})k } .
\end{multline}
In the following proposition we give a $q$-hypergeometric expression for the coefficients $C_n(T_{(2,2t+1)}^{*} ; q)$.
We use the usual characteristic function
\begin{equation*}
  \chi(X) :=
  \begin{cases}
    1, & \text{when $X$ is true,}
    \\
    0, &\text{when $X$ is false.}
  \end{cases}
\end{equation*}

\begin{prop} \label{prop:C_multisum}
We have
\begin{multline}
  \label{C_multisum}
  -q^{t-n} \,
  C_{n-1}(T_{(2,2t+1)}^{*}; q)
  =
  \\ \sum_{n \geq n_{2t-1} \geq
  \cdots \geq n_1 \geq 0} \frac{q^{\sum_{i=1}^{t-1}n_{t+i}^2 +
    \binom{n_t}{2} - \sum_{i=1}^{t-1}n_in_{i+1} - \sum_{i=1}^{t-2}
    n_i}(-1)^{n_t}(1-q^{n_t-\chi(t \geq
    2)n_{t-1}})}{(q)_{n-n_{2t-1}}(q)_{n_{2t-1}-n_{2t-2}}\cdots
  (q)_{n_2-n_1}(q)_{n_1}}.
\end{multline}
\end{prop}



\begin{proof}
In light of equations~\eqref{C_for_Tt} and~\eqref{Baileypairdef1},
we need to find $\beta_n'$ such that $(\alpha_n',\beta_n')$ form a Bailey pair relative to $1$, where 
\begin{equation}
\alpha_n' = q^{(t+1)n^2 -n}(1-q^{2n})\sum_{j=-n}^{n-1}(-1)^jq^{-((2t+1)j^2 + (2t-1)j)/2}.
\end{equation}
We require a result of the second author.  Namely, in part~(ii) of Theorem~1.1 of ~\cite{Lovej14a}, let $k=K=t$, $\ell = t-1$, and $m=0$.   Then $(\alpha_n,\beta_n)$ form a Bailey pair relative to~$1$, where
\begin{equation}
  \beta_n = \sum_{n \geq n_{2t-1} \geq \cdots \geq n_1 \geq 0}
  \frac{q^{\sum_{i=1}^{t-1}n_{t+i}^2+\binom{n_t+1}{2} - \sum_{i=1}^{t-1}n_in_{i+1} - \sum_{i=1}^{t-1}n_i}(-1)^{n_t}}{(q)_{n-n_{2t-1}}(q)_{n_{2t-1}-n_{2t-2}}\cdots (q)_{n_2-n_1}(q)_{n_1}}
\end{equation}
and
\begin{align}
\alpha_n &= q^{(t+1)n^2+n}\sum_{j=-n}^n(-1)^jq^{-((2t+1)j^2 + (2t-1)j)/2} \nonumber \\
&\qquad
- \chi(n \neq 0) q^{(t+1)n^2-n}\sum_{j=-n+1}^{n-1}(-1)^jq^{-((2t+1)j^2 + (2t-1)j)/2} \nonumber \\
&= -q^{(t+1)n^2 -n}(1-q^{2n})\sum_{j=-n}^{n-1}(-1)^jq^{-((2t+1)j^2 + (2t-1)j)/2} \label{alphan'} \\ 
& \qquad
+ \begin{cases} 
1, & \text{if $n=0$}, \\
(-1)^nq^{\frac{n^2}{2} + \left(\frac{2t-3}{2}\right)n} + (-1)^nq^{\frac{n^2}{2} - \left(\frac{2t-3}{2}\right)n}, &\text{if $n \geq 1$}.
\end{cases} \label{alphan''}
\end{align}
Let $\alpha_n = - \alpha_n' + \alpha_n''$ = \eqref{alphan'} +
\eqref{alphan''}.
To find $\beta_n''$, we start with equations~(3.9)
and~(3.10) of~\cite{Lovej14a}, which state that 
\begin{equation}
\alpha_n^* = 
\begin{cases}
1, & \text{if $n = 0$}, \\
(-1)^n\left(q^{(-(2k - 1)n^2 - (2 \ell +1)n)/2} + q^{(-(2k - 1)n^2 + (2 \ell +1)n)/2}\right), & \text{if $n > 0$},
\end{cases}
\end{equation}
and
\begin{equation}
\beta_n^* = \beta_{n_k} = (-1)^{n_k}q^{-\binom{n_k+1}{2}}\sum_{n_k
  \geq n_{k-1} \geq \cdots \geq n_1 \geq
  0}\frac{q^{-\sum_{i=1}^{k-1}n_in_{i+1} -
    \sum_{i=1}^{\ell}n_i}}{(q)_{n_k-n_{k-1}} \cdots
  (q)_{n_2-n_1}(q)_{n_1}} ,
\end{equation}
form a Bailey pair relative to $1$.
Using this Bailey pair with $k=t$ and
$$\ell = \begin{cases} 0,
  & \text{for   $t=1$,}
  \\
  t-2, &  \text{for $t \geq 2$,}
\end{cases}$$
we iterate equations~\eqref{rho1rho2inftyalpha} and~\eqref{rho1rho2inftybeta} $t$ times.  Then $\alpha_n^*$ becomes $\alpha_n''$ and
\begin{equation}
\beta_n'' = \sum_{n \geq n_{2t-1} \geq \cdots \geq n_1 \geq 0} \frac{q^{\sum_{i=1}^{t-1}n_{t+i}^2+\binom{n_t}{2} - \sum_{i=1}^{t-1}n_in_{i+1} - \sum_{i=1}^{t-2}n_i}(-1)^{n_t}}{(q)_{n-n_{2t-1}}(q)_{n_{2t-1}-n_{2t-2}}\cdots (q)_{n_2-n_1}(q)_{n_1}}.
\end{equation} 
Taking $\beta_n'' - \beta_n$ gives the expression for $-q^{t-n}C_{n-1}(T_{(2,2t+1)}^*)$.
\end{proof}


While the above proposition does furnish an attractive
$q$-hypergeometric expression for $C_n(T^*_{(2,2t+1)};q)$, it is not
apparent that these coefficients are Laurent polynomials in~$q$, as
guaranteed by~\eqref{C_nLaurent}.
This is made clear with the next
proposition.

\begin{prop}
  \label{prop:Ct_hypergeometric}
  We have
  \begin{align}
    C_n(T_{(2,2t+1)}^* ; q)
    &=
    q^{n+1-t}
    \sum_{
      n+1=k_t \geq
      k_{t-1} \geq \cdots \geq k_1 \geq 1}
     \prod_{a=1}^{t-1}
    q^{k_a^2} 
    \,
    \frac{
      (q^{1-a+
        \sum_{i=1}^a  2 k_i  }
      )_{k_{a+1}-k_a}
    }{(q)_{k_{a+1} - k_a}}
    \label{Cfunc_simple}
    \\
    &= q^{n+1-t}
    \sum_{
      n+1=k_t \geq
      k_{t-1} \geq \cdots \geq k_1 \geq 1} \prod_{i=1}^{t-1} q^{k_i^2}
    \begin{bmatrix}
      k_{i+1} + k_i - i   + 2\sum_{j=1}^{i-1}k_j \\
      k_{i+1} - k_i
    \end{bmatrix}_q .
    \label{Cfunc_simple2}
  \end{align}
\end{prop}

\begin{proof}
  We recall the classical $q$-binomial identity,
  \begin{equation}
    \sum_{n=0}^N z^nq^{\binom{n}{2}}\begin{bmatrix} N \\ n \end{bmatrix}_q = (z)_N.
  \end{equation}
  Letting $z= -zq^a$, $N = b-a$, and shifting $n$ to $ n-a$, we have the identity
  \begin{gather}
    \label{q_sum_lemma_1}
    \sum_{n=a}^b
    \frac{
      q^{\binom{n}{2}} \, (-z)^n
    }{
      (q)_{b-n} \, (q)_{n-a}
    }
    =
    (-z)^a \, q^{\binom{a}{2}} \,
    \frac{
      (z \, q^a)_{b-a}
    }{
      (q)_{b-a}
    } .
  \end{gather}
  Using this identity, we also have, for arbitrary $c$,
  \begin{equation}
    \label{q_sum_lemma_2}
    \sum_{n=a}^b
    \frac{
      (-1)^n \, (1-q^{n-c}) \, q^{\binom{n}{2} - a n}
    }{
      (q)_{b-n} \ (q)_{n-a}
    }
    =
		\begin{cases}
    (-1)^aq^{-\binom{a+1}{2}}(1-q^{a-c}), & \text{if $a=b$}, \\
		(-1)^{a+1} \, q^{-\binom{a}{2}-c}, & \text{otherwise}.
		\end{cases}
  \end{equation}

We may use the two identities~\eqref{q_sum_lemma_1}
and~\eqref{q_sum_lemma_2} to transform~\eqref{C_multisum}
into~\eqref{Cfunc_simple} as follows.
First, if $n_{t+1} = n_{t-1}$ then the sum in ~\eqref{C_multisum} vanishes, so we
may assume $n_{t+1} > n_{t-1}$.
The sum over $n_t$ is
\begin{equation*}
  \sum_{n_t=n_{t-1}}^{n_{t+1}}
  \frac{(-1)^{n_t}q^{\binom{n_t}{2}-n_{t-1}n_t}(1-q^{n_t-n_{t-1}})}{(q)_{n_{t+1}-n_t}(q)_{n_t-n_{t-1}}},
\end{equation*}
and the second part of identity~\eqref{q_sum_lemma_2} then enables us to evaluate this sum, giving 
\begin{multline}
  -q^{t-n-1}C_n(T_{(2,2t+1)}^*;q)
  =
  \\
  \sum_{n+1\geq n_{2t-1} \geq \cdots
    \geq n_{t+2} \geq n_{t+1}
    > n_{t-1} \geq n_{t-2} \geq \cdots \geq n_1 \geq 0}
  (-1)^{1+n_{t-1}}
  q^{
    \sum_{i=1}^{t-1} n_{t+i}^{~2} - \sum_{i=1}^{t-2} n_in_{i+1} - \sum_{i=1}^{t-1}n_i 
  }
  \\
  \times
  \frac{
    q^{-\binom{n_{t-1}}{2}
    }
  }{
    (q)_{n+1-n_{2t-1}} \cdots (q)_{n_{t+2}-n_{t+1}} 
    (q)_{n_{t-1}-n_{t-2}} \cdots (q)_{n_2-n_1} (q)_{n_1}
  } . \label{sum_nt_Ct}
\end{multline}
We then set $n_{t+1}=n_{t-1}+k_1$, with $k_1 \geq 1$.  The sum over $n_{t-1}$ is
\begin{equation*}
\sum_{n_{t-1}=n_{t-2}}^{n_{t+2}-k_1}
\frac{(-1)^{n_{t-1}}q^{\binom{n_{t-1}}{2}-n_{t-2}n_{t-1}+2k_1n_{t-1}}}{(q)_{n_{t+2}-k_1-n_{t-1}}(q)_{n_{t-1}-n_{t-2}}},
\end{equation*}
and~\eqref{q_sum_lemma_1} allows us to evaluate this sum, resulting in 
  \begin{multline*}
    q^{t-n-1}C_n(T_{(2,2t+1)}^*;q)
    =
    \\
    \sum_{n+1 \geq n_{2t-1} \geq \cdots
      \geq n_{t+2} \geq n_{t-2}+k_1
      > n_{t-2} \geq \cdots \geq n_1 \geq 0}
    (-1)^{n_{t-2}}
    q^{
      \sum_{i=2}^{t-1} n_{t+i}^{~2} - \sum_{i=1}^{t-3}n_in_{i+1} - \sum_{i=1}^{t-2} n_{i}
      }
    \\
    \times
    \frac{ (q^{2k_1})_{n_{t+2} - n_{t-2}-k_1}}{
      (q)_{n_{t+2}- n_{t-2} - k_1}
    } \cdot
    \frac{
      q^{-\binom{n_{t-2}}{2} - n_{t-2} + 2k_1 n_{t-2} +k_1^{~2}
      }
    }{
      (q)_{n+1-n_{2t-1}} \cdots (q)_{n_{t+3}-n_{t+2}} 
      (q)_{n_{t-2}-n_{t-3}} \cdots (q)_{n_2-n_1} (q)_{n_1}
    } .
  \end{multline*}
  We continue in the same manner, next setting $n_{t+2}=n_{t-2}+k_2$, with $k_2 \geq k_1$, and we may then take the sum over
  $n_{t-2}$ using~\eqref{q_sum_lemma_1}.  Iterating this process
  (taking the sum over $n_{t-a}$ after setting $n_{t+a}=n_{t-a}+k_a$), we
  arrive at~\eqref{Cfunc_simple}.  The expression in equation~\eqref{Cfunc_simple2} follows from the fact that
	\begin{equation}
	\begin{bmatrix} n \\ k \end{bmatrix}_q = \frac{(q^{k+1})_{n-k}}{(q)_{n-k}}.
	\end{equation}
\end{proof}




We are now prepared to prove the duality in Theorem \ref{thm:Ft_and_Ut}.

\begin{proof}[Proof of Theorem~\ref{thm:Ft_and_Ut}]
Comparing equations~\eqref{Cfunc_simple} and~\eqref{define_Ut}, we see that 
\begin{equation} \label{define_Ut2}
  U_t (x;q)
  =
  \sum_{n=0}^\infty C_n(T_{(2,2t+1)}^{*}; q) \, 
  (-x \,q)_n 
  (-x^{-1} \,q)_n. 
	\end{equation}
	Therefore, by~\eqref{J_and_C}, $U_t(x;q)$ gives the colored Jones polynomial for $T_{(2,2t+1)}^*$ when $x = q^N$,
	\begin{equation}
    \label{J_and_Ut}
    J_N(T_{(2,2t+1)}^* ; q)
    =
    U_t (-q^{N};q) .
  \end{equation}
	Combining~\eqref{J_and_Ut} with~\eqref{J_and_Ft_unity} and \eqref{mirror_J} gives the statement of the theorem.
\end{proof}

\section{Hecke-Type Formulae}
In this section we prove Theorem~\ref{thm:HeckeUt} as well as a simpler formula when $t=1$.  

\begin{proof}[Proof of Theorem~\ref{thm:HeckeUt}]
In equations~\eqref{rho1rho2inftyalpha} and~\eqref{rho1rho2inftybeta} we let $a=1$, $b=1/c = x$, and apply $\alpha_n'$ and $\beta_n' = -q^{t-n}C_{n-1}$ from the proof of Proposition~\ref{prop:C_multisum}.
Recalling~\eqref{define_Ut2}, we obtain
\begin{align*}
  U_t(-x;q) &= \frac{-q^{-t}(xq)_{\infty}(q/x)_{\infty}}{(q)_{\infty}^2}
  \sum_{n \geq 1} \sum_{k=-n}^{n-1} \frac{(-1)^k q^{(t+1)n^2-(t+\frac{1}{2}) k^2 - (t- \frac{1}{2})k}(1-q^{2n})}{(1-xq^n)(1-q^n/x)} \\
  &= \frac{-q^{-t}(xq)_{\infty}(q/x)_{\infty}}{(q)_{\infty}^2}
  \sum_{n \geq 1} \sum_{k=-n}^{n-1} \frac{(-1)^k
    q^{(t+1)n^2-(t+\frac{1}{2}) k^2 - (t- \frac{1}{2})k}}{1-xq^n}
  \\
  &\hskip1in + \frac{-q^{-t}(xq)_{\infty}(q/x)_{\infty}}{(q)_{\infty}^2}
  \sum_{n \geq 1} \sum_{k=-n}^{n-1} \frac{(-1)^k
    q^{(t+1)n^2+n-(t+\frac{1}{2}) k^2 - (t-  \frac{1}{2})k}}{x(1-q^n/x)} \\
  &= \frac{-q^{-t}(xq)_{\infty}(q/x)_{\infty}}{(q)_{\infty}^2}
  \sum_{n \geq 1} \sum_{k=-n}^{n-1} \frac{(-1)^k q^{(t+1)n^2-(t+\frac{1}{2}) k^2 - (t- \frac{1}{2})k}}{1-xq^n} \\ 
  &\hskip1in -  \frac{-q^{-t}(xq)_{\infty}(q/x)_{\infty}}{(q)_{\infty}^2}
  \sum_{n \leq 1} \sum_{k=n}^{-n-1} \frac{(-1)^k q^{(t+1)n^2-(t+\frac{1}{2}) k^2 - (t- \frac{1}{2})k}}{1-xq^n}. 
\end{align*}
Letting $n=(r+s+1)/2$ and $k=(r-s-1)/2$ in each of the two final sums
leads to the expression in~\eqref{HeckeUteq1}.   The expression
in~\eqref{HeckeUteq2} follows upon expanding the term
$1/(1-xq^{(r+s+1)/2})$ in~\eqref{HeckeUteq1} as a geometric series.   
\end{proof}



When $t=1$, we have the following Hecke-type double sum. 
\begin{theorem}
\begin{equation}
  (1-x) \, U_1 (-x; q) 
  =
  \frac{1}{(q)_{\infty}}\left(\sum_{r,n \geq 0} - \sum_{r,n < 0}\right)(-1)^{n+r}x^{-r}q^{n(3n+5)/2+2nr+r(r+3)/2}.
\end{equation}
\end{theorem}

\begin{proof}
To see this we use the Bailey pair relative to $q$ (Lemma~6 of~\cite{Andre92a}),
\begin{equation*}
\alpha_n = (-x)^{-n}q^{\binom{n+1}{2}}(1-x^{2n+1})
\end{equation*}
and
\begin{equation*}
\beta_n = \frac{(x)_{n+1}(q/x)_n}{(q^2)_{2n}}
\end{equation*}
together with the fact that if $(\alpha_n,\beta_n)$ is a Bailey pair
relative to $a$, then
(Cor.~1.3 of~\cite{Lovej12a})
\begin{equation} \label{conjBailey}
\sum_{n \geq 0} (aq)_{2n}q^n\beta_n = \frac{1}{(q)_{\infty}} \sum_{r,n \geq 0} (-a)^nq^{3n(n+1)/2+(2n+1)r}\alpha_r.
\end{equation}
This gives
\begin{equation}
\sum_{n \geq 0} (x)_{n+1}(q/x)_nq^n = \frac{1}{(q)_{\infty}}\sum_{r,n \geq 0}(-1)^{n+r}x^{-r}q^{n(3n+5)/2+2nr+r(r+3)/2}(1-x^{2r+1}).
\end{equation}
Using $(1-x^{2r+1})$ to split the right-hand side into two sums and then replacing $(r,n)$ by $(-r-1,-n-1)$ in the second sum yields the result.  
\end{proof}

\section{The vector-valued case}




The first author~\cite{KHikami02c} introduced and studied the modular properties of a family of $q$-series $F_t^{(m)}(q)$, defined for
$1\leq m \leq t$ by
\begin{equation}
  \label{eq:3}
  F_t^{(m)}(  q)
  :=
  q^t
  \sum_{k_1, \dots, k_t=0}^\infty
  (q)_{k_t} \,
  q^{k_1^{~2} + \dots + k_{t-1}^{~2} + k_m + \dots + k_{t-1}} \,
  \prod_{i=1}^{t-1}
  \begin{bmatrix}
    k_{i+1} + \delta_{i,m-1} \\
    k_i
  \end{bmatrix}_q .
\end{equation}
The case $m=1$ corresponds to~\eqref{define_Ft},
$F_t(q)=F_t^{(1)}(q)$.   

He showed~\cite{KHikami02c} that at the $N$-th root of unity we have
\begin{equation}
  \label{Ftm_N}
  \zeta_N^{~
    -t + \frac{(2t+1-2m)^2}{8(2t+1)}
  } \,
  F_t^{(m)}(\zeta_N)
  =
  (2t+1) \, N
  \sum_{k=1}^{4(2t+1)N}
  \chi_{8t+4}^{(m)}(k) \,
  \zeta_N^{~ \frac{k^2}{8 (2t+1)} } \,
  B_2
  \left( \frac{k}{4(2t+1)N} \right) ,
\end{equation}
where the Bernoulli polynomial is
$B_2(x)=x^2-x+\frac{1}{6}$, and the periodic function is
\begin{equation}
  \label{define_periodic}
  \chi_{8t+4}^{(m)}(k)
  :=
  \begin{cases}
    1, &
    \text{when $k= \pm(2t+1-2m) \mod 8t+4$},
    \\
    -1, &
    \text{when $k=\pm(2t+1+2m) \mod 8t+4$},
    \\
    0 , &
    \text{otherwise}.
  \end{cases}
\end{equation}
This is a limiting value of the Eichler integral of
a vector modular
form $\Phi_t^{(m)}(\tau)$ with weight~$1/2$,
\begin{equation}
  \begin{aligned}[t]
    \Phi_t^{(m)}(\tau)
    & :=
    \sum_{n=0}^\infty \chi_{8t+4}^{(m)}(n) \,
    q^{\frac{1}{8(2t+1)} n^2}
    \\
    &
    =
    q^{\frac{(2t+1-2m)^2}{8(2t+1)}}
    (q^m, q^{2t+1-m}, q^{2t+1}; q^{2t+1})_\infty
    ,
  \end{aligned}
\end{equation}
where as usual $q=\E^{2\pi\I \tau}$,
and
$(a,b,\cdots; q)_\infty = (a)_\infty (b)_\infty \cdots$.
The quantum modularity of $F_t^{(m)}(q)$ is given by
\begin{multline}
  \label{modular_Ftm}
  \phi_t^{(m)}(z) + \frac{1}{\left( \I \, z \right)^{\frac{3}{2}}}
  \sum_{m^\prime=1}^t
  \frac{2}{\sqrt{2t+1}}
  (-1)^{t+1+m+m^\prime}
  \sin\left( \frac{2 \, m \, m^\prime}{2 t+1} \,  \pi \right)
  \,
  \phi_t^{(m^\prime)} \left(-1/z \right)
  \\
  =
  \frac{\sqrt{(2t+1)\I}}{2\pi}
  \int_0^{\I \infty}
  \frac{
    \Phi_t^{(m)}(w)
  }{
    (w-z)^{\frac{3}{2}}
  } \, dw ,
\end{multline}
where
$z\in \mathbb{Q}$, and
$\phi_t^{(m)}(\tau) := q^{ -t + \frac{(2t+1-2m)^2}{8(2t+1)} } \,
F_t^{(m)}(q)$.
See~\cite{KHikami02c} for details.




In this section we define $q$-series $U_t^{(m)}(q)$ so that 
$F_t^{(m)}(    \zeta_N^{~-1}) = U_t^{(m)}(    -1;     \zeta_N)$
(see Theorem~\ref{thm:Ftm_and_Utm}) and we find Hecke-type formulae for
$U_t^{(m)}(q)$ (see Theorem~\ref{thm:HeckeUtm}).
We begin by defining an analogue of the colored Jones polynomial,
\begin{equation}
  \label{Jpoly_tm}
  (1-q^N) \,
  J_N^{(t,m)}(q)
  :=
  (-1)^N \,
  q^{-t + \frac{N}{2} + \frac{2t+1}{2} N^2}
  \sum_{k=-N}^{N-1}
  (-1)^k \,
  q^{- \frac{2t+1}{2} k(k+1) + m k} .
\end{equation}
When $m=1$ this coincides with
the colored Jones polynomial
$J_N^{(t,1)}(q)=J_N(T_{(2,2t+1)}^*;q)$.

\begin{prop}
    \begin{equation}
      \label{Jtm_F_N}
    J_N^{(t,m)}(\zeta_N)
    =
    F_t^{(m)}(\zeta_N^{~ -1}) .
  \end{equation}
\end{prop}

\begin{proof}
The function $J_N^{(t,m)}(q)$ is also written by use
of the periodic function~\eqref{define_periodic} as
\begin{equation*}
  (1-q^N) \, J_N^{(t,m)}(q)
  =
  (-1)^N    \,
  q^{-t+\frac{1}{2}N + \frac{2t+1}{2}N^2}
  \sum_{k=1}^{2(2t+1)N} 
  \chi_{8t+4}^{(m)}(k) \, q^{- \frac{k^2 - (2t+1-2m)^2}{8(2t+1)}} .
\end{equation*}
At $q\to \zeta_N$,
we have
\begin{align*}
  \zeta_N^{~
    t -\frac{(2t+1-2m)^2}{8(2t+1)}
  } \,
  J_N^{(t,m)}(\zeta_N)
  & =
  - \lim_{q\to \zeta_N} \frac{1}{1-q^N}
  \sum_{k=1}^{2(2t+1)N} \chi_{8t+4}^{(m)}(k) \, q^{- \frac{k^2}{8(2t+1)}}
  \\
  & =
  -\frac{1}{8 (2t+1) N}
  \sum_{k=1}^{2(2t+1)N} k^2 \, \chi_{8t+4}^{(m)}(k) \,
  \zeta_N^{~ -\frac{k^2}{8(2t+1)}} .
\end{align*}
Using~\eqref{define_periodic},
the sum $\sum_{k=0}^{2(2t+1)N}$ may be replaced with
$\frac{1}{2}\sum_{k=-2(2t+1)N}^{2(2t+1)N}$.  Then the right hand side is written as
\begin{equation*}
  -\frac{1}{16(2t+1)N}
  \sum_{\ell=0}^{4(2t+1)N} 
  \left(\ell - 2(2t+1)N\right)^2 \,
  \chi_{8t+4}^{(m)}(\ell - 2(2t+1)N) \,
  \zeta_N^{~ - \frac{\left( \ell - 2 (2t+1) N \right)^2}{8(2t+1)}} .
\end{equation*}
Since
$\chi_{8t+4}^{(m)}(k-2(2t+1)) = - \chi_{8t+4}^{(m)}(k)$,
we obtain
\begin{equation}
  \zeta_N^{~
    t -\frac{(2t+1-2m)^2}{8(2t+1)}
  } \,
  J_N^{(t,m)}(\zeta_N)
  =
  (2t+1) \, N
  \sum_{k=1}^{4(2t+1)N}
  \chi_{8t+4}^{(m)}(k) \,
  \zeta_N^{~- \frac{k^2}{8 (2t+1)} } \,
  B_2
  \left( \frac{k}{4(2t+1)N} \right) .
\end{equation}
Recalling~\eqref{Ftm_N}, the statement follows.
\end{proof}

Next we study a cyclotomic expansion for $J_N^{(t,m)}(q)$,
\begin{equation}
  \label{Jm_C}
  J_N^{(t,m)}(q)
  =
  \sum_{n=0}^\infty C_n^{(t,m)}(q) \, (q^{1+N})_n \, (q^{1-N})_n.
\end{equation}
Equation~\eqref{C_from_J} shows that we have
\begin{equation} \label{Ctm}
  C_{n-1}^{(t,m)}( q    )
  = - q^{n-t} \sum_{\ell=0}^n \frac{1}{(q)_{n-\ell} (q)_{n+\ell}} \,
  q^{(t+1)\ell^2 - \ell} \, ( 1 - q^{2\ell}) \,
  \sum_{k=-\ell}^{\ell-1}
  (-1)^k q^{ - \frac{2t+1}{2} k(k+1) + m k} .
\end{equation}
We note that
$    C_n^{(t,1)}(q) = C_n(T_{(2,2t+1)}^* ; q)$.

The next two propositions are generalizations of Propositions~\ref{prop:C_multisum} and~\ref{prop:Ct_hypergeometric}, respectively.
\begin{prop} \label{prop:Ctm_multisum}
  We have
\begin{equation} \label{Ctm_multisum}
  -q^{t-n}C_{n-1}^{(t,m)} 
  = \sum_{n \geq n_{2t-1} \geq \cdots \geq n_1  \geq 0} 
  \frac{q^{\sum_{i=1}^{t-1}n_{t+i}^2 + \binom{n_t}{2} -   \sum_{i=1}^{t-1}n_in_{i+1} - \sum_{i=1}^{t-m-1}  n_i}(-1)^{n_t}(1-q^{n_t - \chi(t > m)n_{t-m}}) 
  }{
    (q)_{n-n_{2t-1}}\cdots(q)_{n_2-n_1}(q)_{n_1}} .
\end{equation}
\end{prop}

\begin{proof}
In light of equation~\eqref{Ctm} and the definition of a Bailey pair~\eqref{Baileypairdef1}, we need to find $\beta_n'$ such that
\begin{equation}
  \alpha_n' = q^{(t+1)n^2-n}(1-q^{2n})\sum_{k=-n}^{n-1}(-1)^kq^{-(2t+1)k^2/2 - (2t-(2m-1))k/2}.
\end{equation}
The proof of this Bailey pair is exactly as in the proof of Proposition~\ref{prop:C_multisum} but with $\ell = t-m$ instead of $t-1$.
\end{proof}

\begin{prop} We have
  \begin{align}
    C_n^{(t,m)}(q) &=
    q^{n+1-t} \,
    \sum_{\substack{n+1 = k_t \geq k_{t-1}\geq \dots\geq k_1 \geq 0 \\ k_m \geq 1}}
    \prod_{a=1}^{t-1}
    q^{k_a^2}
    \,   \frac{
      (q^{1-a+
        \sum_{i=1}^a \left( 2 k_i + \chi(m>i ) \right) }
      )_{k_{a+1}-k_a}
    }{(q)_{k_{a+1} - k_a}}  \label{def_vector_C} \\
		&= q^{n+1-t} \,   \sum_{\substack{n+1 = k_t \geq k_{t-1}\geq \dots\geq k_1 \geq 0 \\ k_m \geq 1}} \prod_{i=1}^{t-1}
    q^{k_i^{2}} \,
    \begin{bmatrix}
      \displaystyle
      k_{i+1} - k_i - i 
      + \sum_{j=1}^{i} \left( 2 k_j + \chi(m>j) \right)
      \\
      k_{i+1} - k_i
    \end{bmatrix}_q . \label{def2_vector_C}
  \end{align}
\end{prop}

\begin{proof}
 The proof is similar to that of Proposition~\ref{prop:Ct_hypergeometric}.   We begin by treating the sum over $n_t$, which is
\begin{equation} \label{sum_nt}
\sum_{n_t=n_{t-1}}^{n_{t+1}} \frac{(-1)^{n_t}q^{\binom{n_t}{2}-n_{t-1}n_t}(1-q^{n_t-\chi(t > m)n_{t-m}})}{(q)_{n_{t+1}-n_t}(q)_{n_t-n_{t-1}}}.
\end{equation}
Assuming for the moment that $n_{t+1} > n_{t-1}$, the second part of identity~\eqref{q_sum_lemma_2} enables us to evaluate this sum, giving 
\begin{multline}
    -q^{t-n-1}C_n^{(t,m)}(q)
    =
    \\
    \sum_{n+1\geq n_{2t-1} \geq \cdots
      \geq n_{t+2} \geq n_{t+1}
      > n_{t-1} \geq n_{t-2} \geq \cdots \geq n_1 \geq 0}
    (-1)^{1+n_{t-1}}
    q^{
      \sum_{i=1}^{t-1} n_{t+i}^{~2} - \sum_{i=1}^{t-2} n_in_{i+1} - \sum_{i=1}^{t-m}n_i 
    }
    \\
    \times
    \frac{
      q^{-\binom{n_{t-1}}{2}
      }
    }{
      (q)_{n+1-n_{2t-1}} \cdots (q)_{n_{t+2}-n_{t+1}} 
      (q)_{n_{t-1}-n_{t-2}} \cdots (q)_{n_2-n_1} (q)_{n_1}
    } . \label{sum_nt_Cmt}
  \end{multline}
  The rest of the proof is the same as the proof of
  Proposition~\ref{prop:Ct_hypergeometric}.
  The only difference between equations~\eqref{sum_nt_Ct} and~\eqref{sum_nt_Cmt} is that the latter contains the term $q^{-\sum_{i=1}^{t-m} n_i}$ instead of the term $q^{-\sum_{i=1}^{t-1} n_i}$, which results in~\eqref{def_vector_C} instead of~\eqref{Cfunc_simple}.
	
Now suppose that $n_{t+1} = n_{t-1}$.  The sum~\eqref{sum_nt} on $n_t$ is trivial and reduces to 
\begin{equation}
(-1)^{n_{t-1}}q^{-\binom{n_{t-1}+1}{2}}(1-q^{n_{t-1} - \chi(t>m)n_{t-m}}).
\end{equation}
This corresponds to $k_1=0$, and the sum on $n_{t-1}$ is then 
\begin{equation}
\sum_{n_{t-1}=n_{t-2}}^{n_{t+2}} \frac{(-1)^{n_{t-1}}q^{\binom{n_{t-1}}{2}-n_{t-2}n_{t-1}}(1-q^{n_{t-1}-\chi(t > m)n_{t-m}})}{(q)_{n_{t+2}-n_{t-1}}(q)_{n_{t-1}-n_{t-2}}}.
\end{equation}	
If $n_{t-2} = n_{t+2}$ then we collapse the sum again and obtain $k_2=0$, continuing in this way until $n_{t+a} > n_{t-a}$, and then applying~\eqref{q_sum_lemma_1} and arguing as usual.  Note that if $n_{t+m} = n_{t-m}$ then the sum vanishes, so we have 
$k_m \geq 1$.          
\end{proof}


Using the expression for $C_n^{(t,m)}(q)$, we are now prepared to
generalize Definition~\ref{Utdef}.
\begin{definition}
The generalized $U$-function $U_t^{(m)}(x;q)$ is defined by
\begin{equation}
  \begin{aligned}[t]
    U_t^{(m)}(  x; q) 
    & :=
    \sum_{n=0}^\infty
    C_n^{(t,m)}(  q)
    \,
    (-x \, q)_n \, (-x^{-1} q )_n  
    \\
    & =
    q^{-t}
    \sum_{\substack{
        k_t \geq \dots \geq k_1 \geq 0
        \\
        k_m \geq 1
      }}
    (-x q)_{k_t -1}  \,  (-x^{-1} q)_{k_t -1}  \, q^{k_t}
    \\
    & \qquad \qquad \times
    \prod_{i=1}^{t-1}
    q^{k_i^{~2}} \,
    \begin{bmatrix}
      \displaystyle
      k_{i+1} - k_i - i 
      + \sum_{j=1}^{i} \left( 2 k_j + \chi(m>j) \right)
      \\
      k_{i+1} - k_i
    \end{bmatrix}_q .
  \end{aligned}
\end{equation}
\end{definition}

By construction, $U_t^{(m)}(-1;q)$ is dual to $F_t^{(m)}(q)$ as follows.
\begin{theorem} \label{thm:Ftm_and_Utm}
  \begin{equation}
    \label{general_F_U}
    F_t^{(m)}(    \zeta_N^{~-1})
    =
    U_t^{(m)}(    -1;     \zeta_N) .
  \end{equation}
\end{theorem}

\begin{proof}
  We have
  $J_N^{(t,m)}(q) = U_t^{(m)}(-q^N; q)$
  from~\eqref{Jm_C}, thus we get
  \begin{equation}
    J_N^{(t,m)}(\zeta_N) = U_t^{(m)}(-1; \zeta_N) .
  \end{equation}
  With the help of~\eqref{Jtm_F_N}, we get~\eqref{general_F_U}.
\end{proof}

We end this section with the Hecke-type formula for $U_t^{(m)}(x;q)$.   These follow just as those for $U_t(x;q)$ in Theorem~\ref{thm:HeckeUt}, using the Bailey pair in Proposition~\ref{prop:Ctm_multisum} in place of the Bailey pair in Proposition~\ref{prop:C_multisum}
\begin{theorem} \label{thm:HeckeUtm}
  \begin{align}
    U_t^{(m)}(-x;q)
    &=
    -q^{-\frac{t}{2}-\frac{m}{2}+\frac{3}{8}}
    \frac{(x q)_{\infty} ( q/x)_\infty}{
       (q)_\infty ^2
    }
    \\
    &\times
    \Biggl(
      \sum_{\substack{r,s \geq 0 \\ r \not \equiv s \pmod{2}}} -
      \sum_{\substack{r,s < 0 \\ r \not \equiv s \pmod{2}}}
    \Biggr)
    \frac{(-1)^{\frac{r-s-1}{2}}
      q^{\frac{1}{8}r^2+
        \frac{4t+3}{4} r s + \frac{1}{8}s^2+\frac{1+m+t}{2} r
        +
        \frac{1-m+t}{2} s
      }
    }{
      1-x q^{ \frac{r+s+1}{2} }
    } \nonumber \\
		 &=
    -q^{-\frac{t}{2}-\frac{m}{2}+\frac{3}{8}}
    \frac{(x q)_{\infty} (q/x)_\infty}{
      (q)_\infty ^2}
    \\
    &\times
    \Biggl(
      \sum_{\substack{r,s,u \geq 0 \\ r \not \equiv s \pmod{2}}} +
      \sum_{\substack{r,s,u < 0 \\ r \not \equiv s \pmod{2}}}
    \Biggr)
    (-1)^{\frac{r-s-1}{2}}x^u
      q^{\frac{1}{8}r^2+
        \frac{4t+3}{4} r s + \frac{1}{8}s^2+\frac{1+m+t}{2} r
        +
        \frac{1-m+t}{2} s + u\frac{r+s+1}{2}}  .  \nonumber
    \end{align}
\end{theorem}


\section{Concluding Remarks}

In this paper we have studied a family of quantum modular forms $U_t(-1;q)$ based on the colored
Jones polynomial for the torus knot $T_{(2,2t+1)}$.  We have extended the duality between $U_1(-1;q)$ and $F_1(q)$ to general $t$ and determined a Hecke-type expansion for $U_t(x;q)$.   We have further generalized these results to the vector-valued setting.  

We close with two remarks. 

First, in~\cite{BryOnoPitRho12a,AndRhoZwe13a} the modular
transformation formula was given for $U_1(-1;q)$ for generic $q$,
based on an expression for $U_1(x;q)$ in terms of Appell--Lerch series
in~\cite{AndRhoZwe13a}.   The expression in terms of Appell--Lerch
series also shows that $U_1(x;q)$ is a mixed mock modular form for
generic roots of unity $x \neq -1$ and a mock theta function when $x =
\pm i$.   We expect similar results in the general case, and it is to
be hoped that the Hecke series expansions established in this paper
will be useful for determining modular transformation formulae for
$U_t^{(m)}(x;q)$ for generic $q$ and for $x$ a root of unity.
For now, we only know that
by Theorem~\ref{thm:Ftm_and_Utm},
 $U_t^{(m)}(-1;q)$ fulfills~\eqref{modular_Ftm} when $q$ is a root of unity.  

Second, both $U_1(x;q)$ and $F_1(q)$ are interesting combinatorial generating functions.  The two-variable function $U_1(x;q)$ can be interpreted in terms of strongly unimodal sequences and their ranks~\cite{BryOnoPitRho12a,BriFolRho14a}, while $F_1(1-q)$ is the generating function for certain linearized chord diagrams~\cite{DZagie01a} (as well as a number of other objects - see~\cite[A022493]{Sloane} for an overview with references.)  Moreover, there are many nice congruences for the coefficients of $F_1(1-q)$~\cite{AndSel14a} and $U_1(1;q)$~\cite{BryOnoPitRho12a}.  The generalized $U_t^{(m)}(x;q)$ and $F_t^{(m)}(1-q)$ may have interesting combinatorial interpretations and congruence properties, as well.

\section*{Acknowledgements}
KH thanks S.~Zwegers for discussions.
He also thanks organizers
of
``Lectures on $q$-Series and Modular Forms''
(KIAS, July 2013),
``Modular Functions and Quadratic Forms --- number theoretic
delights''
(Osaka University, December 2013),
``Low Dimensional Topology and Number Theory''
(MFO, August 2014).
The work of KH is supported in part by
JSPS KAKENHI Grant Number 23340115, 24654041, 26400079.

\appendix
\section{Examples: $t=2$ and $3$}

Here we record the set of generalized Kontsevich--Zagier series $F_t^{(m)}(q)$ and generalized $U$-functions $U_t^{(m)}(x;q)$ for $t=2$ and $3$. 

When $t=2$, the set of generalized Kontsevich--Zagier series is
\begin{align}
  F_{2}^{(1)}(q)
  & =
  q^2 
  \sum_{n=0}^\infty (q)_n
  \sum_{k=0}^n q^{k(k+1)} \,
  \begin{bmatrix}
    n \\ k
  \end{bmatrix}_q ,
  \\
  F_{2}^{(2)} (q)
  & =
  q^2
  \sum_{n=1}^\infty (q)_{n-1}
  \sum_{k=0}^n
  q^{k^2} \,
  \begin{bmatrix}
    n \\ k
  \end{bmatrix}_q .
\end{align}
The dual $U$-functions,
satisfying
$F_2^{(m)}(\zeta_N^{~ -1})=
U_2^{(m)}(-1; \zeta_N)$,
are given by
\begin{gather}
  \begin{aligned}[t]
    U_{2}^{(1)}(x;q)
    & =
    \sum_{n=0}^\infty
    ( -x q)_n  \, (-x^{-1}q)_n
    q^{n-1} \sum_{k=1}^{n+1} q^{k^2} \,
    \begin{bmatrix}
      n+k \\
      2k-1
    \end{bmatrix}_q
    \\
    & =
    1 + q + (x+2+x^{-1}) q^2
    + (2x+3+2x^{-1}) q^3 + 
    (3x+6+3x^{-1}) q^4+ \cdots
  \end{aligned}
  \\
  \begin{aligned}[t]
    U_{2}^{(2)}(x;q)
    & =
    \sum_{n=0}^\infty
    (-x q)_n (-x^{-1}q)_n \,  q^{n-1}
    \sum_{k=0}^{n+1} q^{k^2} \,
    \begin{bmatrix}
      n+k+1 \\
      2k
    \end{bmatrix}_q
    \\
    & =
    q^{-1} + 2 + (x+2+x^{-1}) q + (2x+4+2x^{-1}) q^2
    + (4x+6+4x^{-1})q^3 + \cdots
  \end{aligned}
\end{gather}

When $t=3$ the set of generalized Kontsevich--Zagier series is 
\begin{align}
  F_{3}^{(1)}(q)
  & =
  q^3 
  \sum_{n=0}^\infty (q)_n
  \sum_{k=0}^n q^{k(k+1)} \begin{bmatrix}
    n \\ k
  \end{bmatrix}_q
	\sum_{j=0}^k q^{j(j+1)}\,
  \begin{bmatrix}
    k \\ j
  \end{bmatrix}_q ,
  \\
F_{3}^{(2)}(q)
  & =
  q^3 
  \sum_{n=0}^\infty (q)_n
  \sum_{k=1}^{n+1} q^{k(k-1)} \begin{bmatrix}
    n \\ k-1
  \end{bmatrix}_q
	\sum_{j=0}^k q^{j^2}\,
  \begin{bmatrix}
    k \\ j
  \end{bmatrix}_q ,
  \\
  F_{3}^{(3)}(q)
  & =
  q^3 
  \sum_{n=1}^\infty (q)_{n-1}
  \sum_{k=0}^n q^{k^2} \begin{bmatrix}
    n \\ k
  \end{bmatrix}_q
	\sum_{j=0}^k q^{j^2}\,
  \begin{bmatrix}
    k \\ j
  \end{bmatrix}_q.
\end{align}

The dual $U$-functions, satisfying
$F_3^{(m)}(\zeta_N^{~ -1})=
U_3^{(m)}(-1; \zeta_N)$,
are given by
\begin{gather}
  \begin{aligned}[t]
    U_{3}^{(1)}(x;q)
    & =
    \sum_{n=0}^\infty
    ( -x q)_n  \, (-x^{-1}q)_n
    q^{n-2} \sum_{k=1}^{n+1} q^{k^2} \,
    \begin{bmatrix}
      n+k-1 \\
      2k-1
    \end{bmatrix}_q
		\sum_{j=1}^{k} q^{j^2} \,
    \begin{bmatrix}
      n+k+2j-1 \\
      2k+2j-2
    \end{bmatrix}_q
    \\
    & =
    1 + q + (x+2+x^{-1}) q^2
    + (2x+4+2x^{-1}) q^3 + 
    (4x+7+4x^{-1}) q^4+ \cdots
  \end{aligned}
  \\
  \begin{aligned}[t]
    U_{3}^{(2)}(x;q)
    & =
    \sum_{n=0}^\infty
    ( -x q)_n  \, (-x^{-1}q)_n
    q^{n-2} \sum_{k=1}^{n+1} q^{k^2} \,
    \begin{bmatrix}
      n+k \\
      2k
    \end{bmatrix}_q
		\sum_{j=0}^{k} q^{j^2} \,
    \begin{bmatrix}
      n+k+2j \\
      2k+2j-1
    \end{bmatrix}_q
    \\
    & =
    q^{-1} + 2 + (x+3+x^{-1}) q
    + (3x+5+3x^{-1}) q^2 + 
    (5x+10+5x^{-1}) q^3+ \cdots
  \end{aligned}
  \\
	\begin{aligned}[t]
    U_{3}^{(3)}(x;q)
    & =
    \sum_{n=0}^\infty
    ( -x q)_n  \, (-x^{-1}q)_n
    q^{n-2} \sum_{k=0}^{n+1} q^{k^2} \,
    \begin{bmatrix}
      n+k \\
      2k
    \end{bmatrix}_q
		\sum_{j=0}^{k} q^{j^2} \,
    \begin{bmatrix}
      n+k+2j+1 \\
      2k+2j
    \end{bmatrix}_q
    \\
    & =
    q^{-2} + 2q^{-1} + (x+3+x^{-1}) 
    + (2x+5+2x^{-1}) q + 
    (5x+8+5x^{-1}) q^2+ \cdots
  \end{aligned}
\end{gather}

\begin{thebibliography}{10}
\providecommand{\url}[1]{\texttt{#1}}
\providecommand{\urlprefix}{URL }
\providecommand{\eprint}[2][]{\url{#2}}

\bibitem{Andre84a}
G.~E. Andrews, \emph{Multiple series {Rogers--Ramanujan} type identities},
  Pacific J. Math. \textbf{114}, 267--283 (1984).

\bibitem{Andre92a}
---{}---{}---, \emph{Bailey chains and generalized {Lambert} series {I}. four
  identities of {Ramanujan}}, Illinois J. Math. \textbf{36}, 251--274 (1992).

\bibitem{AndRhoZwe13a}
G.~E. Andrews, R.~C. Rhoades, and S.~P. Zwegers,
\emph{Modularity of the concave composition generating function},
Algebra \& Number Theory \textbf{7}, 2103--2139 (2013).

\bibitem{AndSel14a}
G.E. Andrews, J.A. Sellers,
\emph{Congruences for the Fishburn numbers},
Preprint. 




\bibitem{BriFolRho14a}
K. Bringmann, A. Folsom, R. Rhoades, \emph{Unimodal sequences and strange functions: a family of quantum modular forms},
Pacific J. Math., to appear. 

\bibitem{BryOnoPitRho12a}
J.~Bryson, K.~Ono, S.~Pitman, and R.~C. Rhoades, \emph{Unimodal sequences and
  quantum and mock modular forms},
  \href{http://dx.doi.org/10.1073/pnas.1211964109}{Proc. Natl. Acad. Sci.}
  \textbf{109}, 16063--16067 (2012).

\bibitem{Habiro00a}
K.~Habiro, \emph{On the colored {Jones} polynomial of some simple links}, RIMS
  Kokyuroku \textbf{1172}, 34--43 (2000).

\bibitem{KHabiro06b}
---{}---{}---, \emph{A unified {Witten}--{Reshetikhin}--{Turaev} invariant for
  integral homology sphere},
  \href{http://dx.doi.org/10.1007/s00222-007-0071-0}{Invent. Math.}
  \textbf{171}, 1--81 (2008),
  \href{http://jp.arxiv.org/abs/math/0605314}{\texttt{arXiv:math/0605314}}.

\bibitem{KHikami04a}
K.~Hikami, \emph{Difference equation of the colored {Jones} polynomial for the
  torus knot}, \href{http://dx.doi.org/10.1142/S0129167X04002582}{Int. J.
  Math.} \textbf{15}, 959--965 (2004),
  \href{http://jp.arxiv.org/abs/math/0403224}{\texttt{arXiv:math/0403224}}.


\bibitem{KHikami02c}
---{}---{}---, \emph{$q$-series and {$L$}-functions related to half-derivatives
  of the {Andrews--Gordon} identity},
  \href{http://dx.doi.org/10.1007/s11139-006-6506-1}{Ramanujan J.} \textbf{11},
  175--197 (2006),
  \href{http://jp.arxiv.org/abs/math/0303250}{\texttt{arXiv:math/0303250}}.

\bibitem{KHikami06b}
---{}---{}---, \emph{Hecke type formula for unified
  {Witten}--{Reshetikhin}--{Turaev} invariant as higher order mock theta
  functions}, \href{http://dx.doi.org/10.1093/imrn/rnm022}{Int. Math. Res. Not.
  IMRN} \textbf{2007}, rnm022--32 (2007).

\bibitem{KHikami03c}
K.~Hikami and A.~N. Kirillov, \emph{Torus knot and minimal model},
  \href{http://dx.doi.org/10.1016/j.physletb.2003.09.007}{Phys. Lett. B}
  \textbf{575}, 343--348 (2003),
  \href{http://jp.arxiv.org/abs/hep-th/0308152}{\texttt{hep-th/0308152}}.

\bibitem{KHikami03b}
---{}---{}---, \emph{Hypergeometric generating function of {$L$}-function,
  {Slater's} identities, and quantum knot invariant},
  \href{http://dx.doi.org/10.1090/S1061-0022-06-00897-1}{Algebra i Analiz}
  \textbf{17}, 190--208 (2005),
  \href{http://jp.arxiv.org/abs/math-ph/0406042}{\texttt{math-ph/0406042}}.


\bibitem{TQLe03a}
T.~T.~Q. Le, \emph{Quantum invariants of 3-manifolds: Integrality, splitting,
  and perturbative expansion},
  \href{http://dx.doi.org/10.1016/S0166-8641(02)00056-1}{Topology Appl.}
  \textbf{127}, 125--152 (2003),
  \href{http://jp.arxiv.org/abs/math/0004099}{\texttt{arXiv:math/0004099}}.

\bibitem{Licko97Book}
W.~B.~R. Lickorish, \emph{An Introduction to Knot Theory}, vol. 175 of
  \emph{Graduate Texts in Mathematics}, Springer, New York, 1997.

\bibitem{Lovej12a}
J.~Lovejoy, \emph{Ramanujan-type partial theta identities and conjugate
  {Bailey} pairs}, \href{http://dx.doi.org/10.1007/s11139-011-9356-4}{Ramanujan
  J.} \textbf{29}, 51--67 (2012).

\bibitem{Lovej14a}
---{}---{}---, \emph{Bailey pairs and indefinite quadratic forms},
  \href{http://dx.doi.org/10.1016/j.jmaa.2013.09.009}{J. Math. Anal. Appl.}
  \textbf{410}, 1002--1013 (2014).

\bibitem{GMasb03a}
G.~Masbaum, \emph{Skein-theoretical derivation of some formulas of {Habiro}},
  Algebraic \& Geometric Topology \textbf{3}, 537--556 (2003).

\bibitem{Mort95a}
H.~R. Morton, \emph{The coloured {Jones} function and {Alexander} polynomial
  for torus knots}, \href{http://dx.doi.org/10.1017/S0305004100072959}{Proc.
  Cambridge Philos. Soc.} \textbf{117}, 129--135 (1995).


\bibitem{RossJone93a}
M.~Rosso and V.~Jones, \emph{On the invariants of torus knots derived from
  quantum groups}, J. Knot Theory Ramifications \textbf{2}, 97--112 (1993).
	
	
\bibitem{Sloane}
N.J.A. Sloane, On-Line Encyclopedia of Integers Sequences, published electronically at \href{http://oeis.org}{\texttt{http://oeis.org}}.

\bibitem{DZagie01a}
D.~Zagier, \emph{Vassiliev invariants and a strange identity related to the
  {Dedekind} eta-function},
  \href{http://dx.doi.org/10.1016/S0040-9383(00)00005-7}{Topology} \textbf{40},
  945--960 (2001).

\bibitem{Zagier09a}
---{}---{}---, \emph{Quantum modular forms}, in E.~Blanchard, D.~Ellwood,
  M.~Khalkhali, M.~Marcolli, H.~Moscovici, and S.~Popa, eds., \emph{Quanta of
  Maths}, vol.~11 of \emph{Clay Mathematics Proceedings}, pp. 659--675, Amer.
  Math. Soc., Providence, 2010.

\end{thebibliography}


\end{document}